\documentclass[11pt, oneside]{article}   	
\usepackage{graphicx}				
\usepackage{amssymb}
\usepackage{amsthm}
\usepackage{amsmath}

\newcommand{\sqrtpi}{\sqrt{\pi}}
\newcommand{\twoOverSqrtpi}{\frac{2}{\sqrtpi}}

\newcommand\Arg[1]{{\rm Arg}\,#1}
\renewcommand\Re[1]{{\rm Re}\, #1}

\newcommand\erf[1]{{\rm erf}\, #1}
\newcommand\erfc[1]{{\rm erfc}\, #1}
\newcommand{\cee}{\mathbb{C}}
\newcommand{\at}{\makeatletter @\makeatother}
\newtheorem{proposition}{Proposition}

\title{The real part of the complementary error function}
\author{Yossi Lonke}
\date{}							

\begin{document}
\begin{titlepage}
  \maketitle
  \begin{abstract}
It is shown that the real part of the complementary error function
    $$\erfc(z)=\twoOverSqrtpi\int_z^{\infty}e^{-u^2}\,du$$
    satisfies $\Re(\erfc(z))\geq 1$ in the set $S=\{z:|\Arg z|\geq 3\pi/4\}$. This improves a previous result asserting that $\erfc(z)$ has no zeros in the set~$S$. 
   \end{abstract}
   \vspace{1 cm}
\vfill
Yossi Lonke\\
yossilonke\at me.com \\
Independent Research \\
13 Basel Street, Tel-Aviv ,Israel\\
Orcid Id: 0000-0001-7493-5085\\
\end{titlepage}

\section*{}

In \cite{EL2008}, it was shown that the complementary error function, namely,
\begin{equation}\label{eq:erfc}
\erfc(z)=\twoOverSqrtpi\int_z^{\infty}e^{-u^2}\,du
\end{equation}
does not vanish in the set
$$S=\{z\in\cee: |\Arg z|\geq\frac{3\pi}{4}\}$$
(ibid. Theorem~1). Here $\Arg z$ denotes the principal argument in the complex plane, i.e. the angle between the complex number $z$ and the positive real axis, taking values in the interval $(-\pi,\pi]$. The proof presented in \cite{EL2008} is essentially a
real-analysis argument. In particular, no use was made of the fact that the complementary error function is analytic.
By combining of a result from \cite{strand1965} and some elementary complex analysis,  we obtain a somewhat simpler proof than the one in \cite{EL2008} , of the following stronger result.
\begin{proposition} For every $z\in S$ one has
\begin{equation}\label{eq:prop}
    \Re(\erfc(z))\geq 1
\end{equation}
Equality holds if and only if $z=0$.
\end{proposition}
\begin{proof}
It clearly suffices to show that  $\erf(z)=1-\erfc(z)$ has a non-negative real part in~$S$,
where
\begin{equation}\label{eq:erf}
\erf(z)=\twoOverSqrtpi\int_0^ze^{-u^2}\,du
\end{equation}
Since $\erf(-z)=-\erf(z)$, this is equivalent to showing that $\Re(\erf(z))\geq 0$ in the set
$$T=-S=\{z\in\cee:|\Arg z|\leq\pi/4\}$$ 
The following estimate for the complementary error function was derived in \cite{strand1965},  (ibid. (10))
\begin{equation}\label{strand}
|\erfc(x+iy)|<\frac{e^{y^2-x^2}}{x\sqrtpi}\sqrt{1+y^2/x^2}\quad (x,y>0)
\end{equation}
Since $\erfc(\bar{z})=\overline{\erfc(z)}$ for all $z\in\cee$, the estimate (\ref{strand}) holds for $y<0$ as well.
If ${z=x+iy\in T}$, then $y^2\leq x^2$, and if also $x>\sqrt{2/\pi}$ then (\ref{strand}) implies that ${|1-\erf(z)|<1}$, and in particular $\Re(\erf(z))>0$.

Consider therefore a complex number $z=x+iy\in T$ such that ${x\leq\sqrt{2/\pi}}$, and  we may further assume that ${y\geq 0}$, because $\erf(\bar{z}),\erf(z)$ are complex conjugates having the same real part. Put $a^2=x^2-y^2$. Since $\erf z$
is an entire function, the integral defining $\erf z$, namely (\ref{eq:erf}),
does not depend on any particular choice of a continuous path with which we choose to connect the origin to $z$. Choose a path composed of a segment connecting the origin to the point $(a,0)$, followed by portion of the curve $x^2-y^2=a^2$, connecting the point $(a,0)$ to the point $(x,y)$. A parametrization $\gamma:[0,x]\to T$ is given by
\begin{equation}\label{eq:gammaPath}
\gamma(t)=
\left\{
\begin{aligned}
&t,& 0\leq t\leq a,\\
&t+i\sqrt{t^2-a^2},& a\leq t\leq x\\
\end{aligned}
\right.
\end{equation}
Evaluating the integral (\ref{eq:erf}) along the first part of $\gamma(t)$ gives 
\begin{equation}\label{easyPart}
\twoOverSqrtpi\int_0^ae^{-u^2}\,du=\erf(a)
\end{equation}
Along the second part of the path, the corresponding integral is
$$
\twoOverSqrtpi\int_{a}^{x}e^{-(t+i\sqrt{t^2-a^2})^2}(1+i\frac{t}{\sqrt{t^2-a^2}})\,dt
$$
The real part of this integral is
\begin{equation}\label{realPart}
\frac{2e^{-a^2}}{\sqrtpi}\int_a^{x}(\cos 2t\sqrt{t^2-a^2}+\sin 2t\sqrt{t^2-a^2}\frac{t}{\sqrt{t^2-a^2}})\,dt
\end{equation}
By hypothesis, $x\leq\sqrt{2/\pi}$. Hence
$$0\leq 2t\sqrt{t^2-a^2}\leq 4/\pi<\pi/2\qquad (a\leq t\leq x)$$
Since both sine and cosine are non-negative in $[0,\pi/2]$, the integral in (\ref{realPart}) is non-negative. Since $\Re(\erf(z))$ is the sum of (\ref{easyPart}) and (\ref{realPart}), this completes the proof that $\Re(\erf(z))\geq 0$ for all $z\in T$.

In case of equality in (\ref{eq:prop}) for
some $z=x+iy\in S$, we switch to ${-z\in T}$, and note that the considerations above imply that $|x|\leq\sqrt{2/\pi}$. As a result, $\Re(\erf(z))=0$ implies that both integrals 
(\ref{easyPart}) and (\ref{realPart}), (the latter with $x$ replaced by $-x=|x|$) must vanish, which forces $x=y=0$.
This completes the proof of the proposition.
\end{proof}
\noindent{\bf Remark} Note that if $x=y$ then $a=0$. In that case the path $\gamma$ reduces to the line segment connecting the origin
to the point $(x,x)$, and the integral (\ref{realPart}) reduces to an integral that was analysed in \cite{EL2008} by using other methods. The proof above, however, remains unchanged. 

\end{document}